\documentclass[a4paper,12pt]{article} 
\title{Separated quotients of Picard schemes}

\usepackage{amsmath, amssymb, mathrsfs, amsthm, geometry, verbatim, mathtools}
\usepackage{hyperref}
\usepackage[all]{xy}
\usepackage[OT2, T1]{fontenc}

\usepackage{pgf,tikz}
\usetikzlibrary{arrows}

\usepackage{cleveref}
\usepackage{tikz-cd}

\newcommand{\on}[1]{\operatorname{#1}}

\theoremstyle{definition}
\newtheorem{definition}{Definition}[section]

\theoremstyle{plain}

\newtheorem{lemma}[definition]{Lemma}
\newtheorem{theorem}[definition]{Theorem}
\newtheorem{corollary}[definition]{Corollary}

\theoremstyle{remark}

\newtheorem{remark}[definition]{Remark}

\renewcommand{\phi}{\varphi}

\newcounter{temp}

\author{Garnet Akeyr}
\date{\today}

\newcounter{nootje}
\newcommand{\beq}{\begin{equation}}
\newcommand{\eeq}{\end{equation}}
\newcommand{\beqs}{\begin{equation*}}
\newcommand{\eeqs}{\end{equation*}}

\begin{document}
\maketitle
\begin{abstract} 
We give some necessary and sufficient conditions for the existence of N\'{e}ron models of jacobians of semistable morphisms of arbitrary relative dimension over base schemes of arbitrary dimension. To do this, we introduce a notion of alignment for semistable morphisms over any regular base scheme, and show that the jacobian of an aligned projective semistable morphism admits a separated model with the N\'{e}ron mapping property. When the Picard scheme is smooth over the base scheme along its unit section we show that the converse holds.
\end{abstract}

\newcommand{\et}{^{\on{et}}}

\section{Introduction}\label{sec:one}
N\'{e}ron models were introduced in 1964 by N\'{e}ron \cite{neron} for abelian varieties over the fraction field of a Dedekind scheme. Their existence has been used to show seminal results in the study of abelian varieties, such as the Serre-Tate theorem on good reduction of abelian varieties \cite{serre-tate}. Raynaud \cite{raynaud} showed in 1970 that N\'{e}ron models are well-suited to the study of Jacobians of curves: given a family of curves $X$ over a Dedekind scheme $S$, smooth over a dense open subscheme $U\subset S$, the N\'{e}ron model of the relative Jacobian $J(X_U)$ of $X_U/U$ is a quotient of the relative Jacobian $J(X)$ of $X/S$. Contemporary results have expanded on this theory to the case of an arbitrary regular base scheme \cite{holmes}, nodal curves over traits \cite{giulio}, and constructing universal N\'{e}ron models for the moduli stack $\mathcal{M}_{g,n}$ of $n$-marked genus $g$ curves \cite{melo}, among other examples. Recall the definition of a N\'{e}ron model:

\begin{definition}
Let $S$ be a scheme, $U\subset S$ a dense open subscheme, and $A/U$ an abelian scheme. A N\'{e}ron model of $A/U$ over $S$ is a smooth, separated algebraic space $N$ over $S$ together with an isomorphism $N_U \cong A$, that satisfies the N\'{e}ron mapping property: Given any smooth algebraic space $T/S$ along with a morphism $f_U : T_U\rightarrow A$ over $U$, there exists a unique extension of $f_U$ to a morphism $f: T \rightarrow N$ over $S$
\end{definition}

\begin{remark}
The usual definition of N\'{e}ron models also requires them to be of finite type over $S$. The above definition is more commonly referred to as a N\'{e}ron \textit{lft} model, where the \textit{lft} stands for locally of finite type. We shall not use this terminology. 
\end{remark}

Let $X\rightarrow S$ be a semistable curve, where $S$ is a regular base scheme. That is, $X/S$ is proper, flat, and the fibres are semistable curves. Let $U\subset S$ be a dense open subscheme, and suppose furthermore that $X$ is smooth over $U$. Then Holmes showed in Section 5 of \cite{holmes} that the Jacobian of $X_U/U$ admits a N\'{e}ron model over $S$ if and only if the geometric fibres of $X/S$ are aligned, a combinatorial condition on the dual graphs of the geometric fibres.

The aim of this paper is to generalise this to higher dimensional semistable morphisms. One must replace the Jacobian by $\text{Pic}^0_{X_U/U}$, the connected component of the unit section of the Picard space of $X/S$. We extend the notion of alignment to higher-dimensions in Section \ref{sec:two}, and that alignment of a semistable morphism implies the existence of a separated quotient of the Picard scheme satisfying the N\'{e}ron mapping property.

\begin{theorem} [Theorem \ref{mainresult1} and Theorem \ref{mainresult2}]
Let $X\rightarrow S$ be a semistable morphism over a regular base scheme $S$, and assume $X$ is regular. Then $X\rightarrow S$ is aligned if and only if the closure $\text{clo}(e)$ of the unit section of $\text{Pic}^0_{X_U/U}$ in $\text{Pic}_{X/S}$ is \'{e}tale over $S$.
\end{theorem}

Furthermore, if the Picard scheme of $X/S$ is smooth along the unit section, such as when one restricts to characteristic $0$, and if $X/S$ is projective then one has a much stronger result.

\begin{theorem}[Theorem \ref{mainresult2}]
Let $X\rightarrow S$ be a projective semistable morphism with $X$ regular and $S$ an excellent, regular, locally Noetherian scheme. Assume also that $\text{Pic}_{X/S}$ is smooth over $S$ along the unit section. Then a N\'{e}ron model for $\text{Pic}^0_{X_U/U}$ exists if and only if $X/S$ is aligned.
\end{theorem}

\subsection{Idea of proof}
One shows that the Picard space $\text{Pic}_{X/S}$ satisfies the existence part of the N\'{e}ron mapping property. Separatedness would imply uniqueness, but this is rarely the case: the failure of $\text{Pic}_{X/S}$ to be separated is equivalent to the failure of the unit section to be a closed immersion. If the closure of the unit section $\text{clo}(e)$ is flat over $S$, one may construct a quotient of $\text{Pic}_{X/S}$ by $\text{clo}(e)$ that is separated and satisfies the N\'{e}ron mapping property. 

It is shown in \cite{holmes} that, for semistable curves, the closure of the unit section is flat over $S$ if and only if $X/S$ is aligned. Namely, alignment is a condition that allows one to construct certain Cartier divisors on $X/S$ that correspond to points in $\text{clo}(e)$. We construct a notion of alignment for higher dimensional varieties that generalises the case of curves. This relies on a stratification of the fibres of $X/S$ by certain subschemes. We then construct maps between the dual graphs of the geometric fibres, and finally show that this generalisation is equivalent to the flatness of $\text{clo}(e)$. 

Finally, we show that the smoothness of the N\'{e}ron model is ensured if $\text{Pic}_{X/S}$ is smooth along its unit section. The idea is to show that the smooth locus of $\text{Pic}_{X/S}$ is an open and closed subgroup scheme containing $\text{Pic}^0_{X/S}$. In this case one then shows that if a N\'{e}ron model exists, then the closure of the unit section is flat over $S$. 

\subsection{Overview of paper}
We introduce the notions of semistable morphisms and the dual graphs of their fibres in Section \ref{sec:two}, as well as introduce the notion of alignment of a labelled graph. Section \ref{sec:three} studies at how the fibres of semistable morphisms are related under specialisation and generalisation of the points in $S$. The main result here is the construction of a map of dual graphs $\text{sp}:\Gamma_s\rightarrow \Gamma_{\eta}$ over points $s$ and $\eta$ in $S$, where $\eta$ is a generalisation of $s$ and $\Gamma_p$ is the dual graph of the fibre $X_p$ for a point $p\in S$. This is used in Section \ref{sec:four} to show that certain Weil divisors on $X/S$ are also Cartier divisors, and that furthermore these Cartier divisors only exist when $X/S$ is aligned at all geometric points. In Section \ref{sec:five} we show our notion of alignment is equivalent to the flatness of the unit section in the Picard scheme of $X/S$, which in turn allows us to construct a separated quotient of the Picard scheme satisfying the N\'{e}ron mapping property.

\subsection{Acknowledgements}
The author would like to thank their supervisor, David Holmes, for introducing them to this problem, as well as for his guidance and insights on the problem. The author would also like to thank Bas Edixhoven, Raymond van Bommel, and Giulio Orecchia for many helpful conversations and discussions on N\'{e}ron models, as well as David Rydh for his suggestion to use Lemma (C.2) of his paper \cite{rydh} in the proof of \hyperref[mainresult1]{Theorem \ref*{mainresult1}}, and Matthieu Romagny for his insight on \hyperref[romagnylemma]{Lemma \ref*{romagnylemma}}.

\section{Semistable morphisms and alignment}\label{sec:two}
\newcommand{\mono}{\textbf{Mon}_0}

\subsection{Definition of semistable morphism}
\begin{definition}
\label{semistablemorphism}
Let $S$ be a locally Noetherian scheme. A morphism $f: X\rightarrow S$ is a \textit{semistable morphism} if it is proper with geometrically connected fibres, and such that the following condition holds: For all $x\in X$ with image $s\in S$, there exists

\begin{itemize}
\item an \'{e}tale neighbourhood $(\text{Spec}(R), s')$ of $(S,s)$;
\item an \'{e}tale neighbourhood $(U,x')$ of $(X,x)$ where $U$ is a connected scheme;
\item an element $b\in R$;
\item an \'{e}tale morphism $U\rightarrow \text{Spec}(R[x_1, \ldots, x_n] / (x_1\cdots x_l - b))$,
\end{itemize}

making the following diagram commute:

\begin{equation}\label{semistablediagram}
\begin{tikzcd}
& U \arrow[dl, "\text{\'{e}t}" above] \arrow[dr, "\text{\'{e}t}"] & \\
\text{Spec}(R[x_1, \ldots, x_n] / (x_1\cdots x_l - b)) \arrow[d] &    & X \arrow[d, "f"] \\
\text{Spec}(R) \arrow[rr, "\text{\'{e}t}"] & & S
\end{tikzcd}
\end{equation}
\end{definition}

\begin{definition}
Let $f: X\rightarrow S$ be a semistable morphism. Let $x\in X$ be a point with image $s\in S$, and suppose we have a diagram as in  (\ref{semistablediagram}), where $(U,x')$ is an \'{e}tale neighbourhood of $(X,x)$. Let $p$ be the image of $x'$ in $\text{Spec}(R[x_1, \ldots, x_n] / (x_1\cdots x_l - b))$. If $\{x_1, \ldots, x_l\}$ are contained in the prime ideal corresponding to $p$, we say that the data of a diagram as in (\ref{semistablediagram}) is a \textit{local chart of $x$}.

We use $U\rightarrow \text{Spec}(R[x_1, \ldots, x_n] / (x_1\cdots x_l - b))$ to denote the data of a local chart of $x$. 
\end{definition}

\begin{remark}
The definition of local chart is a special case of that given in definition 1.1 of \cite{li}. 
\end{remark}

 The next lemma is immediately verified by the reader.
 
\begin{lemma}
If $f: X\rightarrow S$ is a semistable morphism, then every point $x\in X$ admits a local chart.
\end{lemma}



If $U\rightarrow \text{Spec}(R[x_1, \ldots, x_n] / (x_1\cdots x_l - b))$ is a local chart of a point $x$, the following lemma of Li shows that the integer $l$ is independent of the local chart.

\begin{lemma}[\cite{li}, Lemma 3.2]
\label{lilabellemma}
Let $f: X \rightarrow S$ be a semistable morphism. Let $x\in X$, $s = f(x)$, and let 
\[U \rightarrow \text{Spec}(R[x_1, \ldots x_n]/(x_1\cdots x_l -b))\]
and 
\[U' \rightarrow \text{Spec}(R'[x_1, \ldots x_n]/(x_1\cdots x_{l'} -b'))\]
be two local charts of $f$ at $x$. Then
\begin{itemize}
\item $l = l'$,
\item the images of $b$ and $b'$ in $\mathcal{O}_{S,s}^{\text{sh}}$ satisfy $b = ub'$ for some unit $u\in \mathcal{O}_{S,s}^{\text{sh}}$.
\end{itemize}
Here $\mathcal{O}_{S,s}^{\text{sh}}$ denotes the strict henselisation of $\mathcal{O}_{S,s}$ for a fixed separable closure of $k(s)$. 
\end{lemma}

\begin{remark}
When the fibres of $X/S$ are one-dimensional, one can show as in Proposition 2.5 of \cite{holmes} that if $f: X\rightarrow S$ is proper and flat whose fibres are semistable curves, then the \'{e}tale local rings of closed points in the non-smooth locus are of the form $\mathcal{O}_{S,\overline{s}}[[x,y]]/(xy-b)$, where $b$ is in the maximal ideal of $\mathcal{O}_{S,\overline{s}}$. The argument used for showing this does not generalise to higher dimensions, and so we impose the condition of semistability as in the Definition \ref{semistablemorphism} to ensure the \'{e}tale local rings have a similar form as in the case of curves.
\end{remark}

\begin{definition}
Let $f:X\rightarrow \text{Spec}(k)$ be a semistable morphism with $k$ a field, and let $\text{Sing}_f^0(X) = X$, and inductively define $\text{Sing}_f^k(X)$ as the non-smooth locus of $\text{Sing}_f^{k-1}(X)$, viewed as a closed subscheme of $X$ with the reduced induced subscheme structure. The \textit{$k$-strata of $X$} are the irreducible components of the smooth locus of $\text{Sing}_f^k(X)$.
\end{definition}

\begin{remark}\label{smoothlocusisdense}
Let $f: X\rightarrow S$ be a semistable morphism, $s\in S$ a point, and $X_s$ the fibre of $X$ over $s$. The property of being semistable is preserved under base change, and so $f: X_s\rightarrow s$ is a semistable morphism. If $X_s$ is of the form $k[x_1, \ldots, x_n]/(x_1\cdots x_l)$, then the non-smooth locus of $\text{Sing}_f^k(X)$ is of codimension $1$. In the general case it suffices to work \'{e}tale locally and conclude the same. The closure of a $k$-strata is an irreducible component of $\text{Sing}_f^k(X)$, and conversely the smooth locus of an irreducible component of $\text{Sing}_f^k(X)$ is a dense open subscheme. Hence there is a bijection between elements of the $k$-strata and irreducible components of $\text{Sing}_f^k(X)$.
\end{remark}

\begin{lemma}\label{compsareirr}
Let $k$ denote a separably closed field, and let $X\rightarrow\text{Spec}(k)$ be a semistable morphism. Then the elements of the $k$-strata of $X$ are geometrically irreducible. 
\end{lemma}
\begin{proof}
It suffices by [\cite{stacks}, Tag 04QM] to show that each irreducible component of $\text{Sing}_f^k(X)$ admits a rational point in its smooth locus.

If $X = \text{Spec}(k[x_1, \ldots, x_n] / (x_1\cdots x_l))$, then the  irreducible components of $\text{Sing}_f^r(X_s)$ are cut out by ideals of the form $(x_{i_1}, \ldots, x_{i_{r+1}})\subset \{x_1, \ldots, x_l\}$, and it is clear that rational points exist in the smooth locus of any $r$-strata. 

In the general case, any closed point $x$ of an element of the $r$-strata admits a local chart $U\rightarrow \text{Spec}(k[x_1, \ldots, x_n] / (x_1\cdots x_{r+1}))$. The image of $U$ is open and contains the point $(x_1, \ldots, x_{r+1})$, hence we can find a $k$-rational point $p$ in the image of $U$ whose ideal contains $(x_1, \ldots, x_{r+1})$ and is contained in an element of the $r$-strata.

As $k$ is separably closed, any pre-image of $p$ in $U$ is $k$-rational and smooth over $S$, whence its image in $X$ is also $k$-rational and smooth over $S$, necessarily in the same element of the $r$-strata of $x$ by construction. Hence we conclude that the element of the $r$-strata containing $x$ is geometrically irreducible.
\end{proof}

\subsection{Graphs and alignment}
Graphs will play a key role in the notion of alignment of a semistable morphism. We recall some basic definitions.
\begin{definition}
A \textit{graph} is an ordered triple $\Gamma = (V,E, \text{Incidence})$ where $V$ is a finite set (called the \textit{vertices} of the graph $\Gamma$), $E$ the (finite) set of \textit{edges}, and $\text{Incidence}:E \rightarrow (V\times V)/S_2$. Intuitively, the incidence relations tell us which vertices an edge connects. A graph $\Gamma$ is \textit{connected} if there is a path between any two vertices of $\Gamma$. A \textit{cycle} of $\Gamma$ is a path whose first and last vertex coincide and for which no other vertices or edges are repeated.
\end{definition}

\begin{definition}
Let $X/S$ be a semistable morphism, and let $X_s$ be the fibre of $X$ over a point $s\in S$. The \textit{dual graph} $\Gamma_{s}$ of  $X_{s}$ is the graph $\Gamma_{s}$ defined as follows: For every irreducible component $X_i$ of $X_{s}$ we have one vertex $v_i$. For every $1$-strata $C_e$ contained exclusively within components $X_i$ and $X_j$ (where possibly $X_i=X_j$) we have one edge between vertex $v_i$ and $v_j$.
\end{definition}

\begin{remark}
In general the irreducible components and the $1$-strata of a fibre are not geometrically irreducible, and so given a geometric point $\overline{s}\rightarrow S$ with image $s$ it is not in general true that $\Gamma_s$ will be isomorphic to $\Gamma_{\overline{s}}$, where $\Gamma_{\overline{s}}$ is the dual graph of the geometric fibre $X_{\overline{s}}$. 

However, by \hyperref[compsareirr]{Lemma \ref*{compsareirr}}, the dual graph $\Gamma_{\overline{s}}$ is equivalent to the dual graph defined using elements of the $0$- and $1$- strata of $X_{k_s^{\text{sep}}}$, where $k_s^{\text{sep}}$ is the separable closure of $k(s)$ within $k(\overline{s})$.
\end{remark}

\begin{remark}
This generalises the dual graph of a semistable curve as defined in \cite{holmes}. Namely, in the case of a semistable curve, each irreducible component has an associated vertex, and for each point of intersection between irreducible components $C_i$ and $C_j$ one has an edge between the associated vertices. The set of points of intersection in this case constitute the $1$-strata.
\end{remark}

\subsection{Alignment of a semistable morphism}
Having defined the strata of any fibre $X_s$ of a semistable morphism $f:X\rightarrow S$, we shall attach labels to the strata in such a way that the labels are constant on any connected component of $\text{Sing}_f(X_s)$. These labels are the analogues of those in Definition 2.11 of \cite{holmes}. 

\begin{definition}\label{edgelabelling}
Let $x\in X_s$ be in the non-smooth locus. Let \[g: U \rightarrow \text{Spec}(R[x_1, \ldots x_n]/(x_1\cdots x_l -b)\] be a local chart of $x$. We define the label of $x$ to be $l(x) = (b)$, viewed as a principal ideal of $\mathcal{O}_{S,s}^{\text{sh}}$.
\end{definition}

\begin{remark}
By \hyperref[lilabellemma]{Lemma \ref*{lilabellemma}} this definition is independent of the local chart we choose.
\end{remark}

\begin{lemma}[\cite{li}, Lemma 3.2 and Lemma 3.4]
The label $l(x)$ as in the above definition is constant on connected components of the non-smooth locus of $X_s$. 
\end{lemma}

\begin{remark}
While the label $l(x)$ is constant on connected components of $\text{Sing}_f(X_s)$, the same is not true of the integer $l$ in \hyperref[lilabellemma]{Lemma \ref*{lilabellemma}}. For $x \in \text{Sing}_f^k(X_s) \setminus \text{Sing}^{k+1}_f(X_s)$ we will have $l=k+1$, and hence $l$ will in general vary on a given connected component $C$ of $\text{Sing}_f(X_s)$.
\end{remark}

\begin{definition}
Let $Y$ be a $k$-strata of a fibre $X_s$, $k\geq 1$, of a semistable morphism $X/S$. The \textit{label of $Y$} is defined to be the label $l(x)$ of any point $x\in Y$. 
\end{definition}

We now define a labelling of the dual graphs of geometric fibres of $X/S$. Using these labels we can attach labels to edges of the dual graph $\Gamma_{\overline{s}}$ of $X_{\overline{s}}$, where $\overline{s}\rightarrow s$ is a geometric point.

\begin{definition}
Let $\Gamma$ be a graph with vertex set $V$ and edge set $E$. Given a monoid $M$, an \textit{edge-labelling} of $\Gamma$ is a function $f: E \rightarrow M$. We can similarly define a \textit{vertex labelling} as a function $g: V \rightarrow M$. 

Let $C$ be a cycle of a graph $\Gamma$, with an edge-labelling $f: E \rightarrow M$. We say that $C$ is \textit{aligned} with respect to the edge-labelling if for every pair of edges $e_1$ and $e_2$ in $C$ there exist positive integers $n_1$ and $n_2$ such that $f(e_1)^{n_1} = f(e_2)^{n_2}$. We say that an edge-labelled graph $(\Gamma, f)$ is aligned if every circuit of $\Gamma$ is aligned with respect to $f$.
\end{definition}

Recall from \hyperref[compsareirr]{Lemma \ref*{compsareirr}} that the $0$- and $1$- strata of $X_{\overline{s}}$ are in bijection with the $0$- and $1$-strata of the special fibre of $X\times_S \mathcal{O}_{S, \overline{s}}$ over $\mathcal{O}_{S, \overline{s}} = \mathcal{O}_{S,s}^{\text{sh}}$ for any geometric point $\overline{s}\rightarrow S$, with the closed point of the latter corresponding to the separable closure of $k(s)$ in $k(\overline{s})$.
\begin{definition}
\label{defofalignment}
Let $X/S$ be a semistable morphism and let $\overline{s}\rightarrow s$ be a geometric point of $S$. Let $\Gamma_{\overline{s}}$ be the dual graph of $X_{\overline{s}}$, and let $M_s$ be the monoid of principal ideals of the ring $\mathcal{O}_{S,\overline{s}}$. We define a labelling of the edges of $\Gamma_{\overline{s}}$ by sending an edge $e$ to the label of the $1$-strata of the closed fibre of $X\times_S \mathcal{O}_{S, \overline{s}}$ over $\mathcal{O}_{S,\overline{s}}$ that $e$ corresponds to.

We say that \textit{$X/S$ is aligned at $\overline{s}$} if the dual graph $\Gamma_{\overline{s}}$ is aligned with respect to the above labelling. We say that $X/S$ is \textit{aligned} if it is aligned at all geometric points of $S$.
\end{definition}

\section{Behaviour of strata under specialisation  and generalisation}\label{sec:three}

In this section we will study how the dual graph $\Gamma_{\overline{s}}$ of a geometric fibre $X_{\overline{s}}$ of a semistable morphism is related to the dual graph of $\Gamma_{\eta}$, where $\eta$ is a generalisation of $s$. The main result is that we will obtain a map $\text{sp}: \Gamma_s \rightarrow \Gamma_{\overline{\eta}}$.

\subsection{Specialisation map on vertices}

\hyperref[compsclosedfromgeneric]{Lemma \ref*{compsclosedfromgeneric}} and \hyperref[nosharedirred]{Lemma \ref*{nosharedirred}} hold in more general situations than the case of semistable morphisms when the base scheme is integral.

\begin{lemma}\label{compsclosedfromgeneric}
Let $f: X\rightarrow S$ be a proper, flat morphism with reduced fibres to a locally Noetherian integral scheme $S$ with generic point $\eta$. Assume that the fibres of $f$ are pure dimensional. Let $Y\subset X_{\eta}$ be an irreducible component of $X_{\eta}$ and let $\overline{Y}$ denote the scheme-theoretic closure of $Y$ in $X$. Then for every point $s\in S$, $\overline{Y}_s$ is a union of irreducible components of $X_s$. Moreover, for every irreducible component $Z$ of $X_s$ there exists an irreducible component $Y$ of $X_{\eta}$ such that $Z$ is contained in $\overline{Y}_s$.
\end{lemma}

\begin{proof}
We will show the dimension of each irreducible component of $\overline{Y}_s$ is equal to the dimension of $X_s$ and hence that each irreducible component of $\overline{Y}_s$ is an irreducible component of $X_s$. To achieve this, we will show
\[\text{dim}(X_s) \geq \text{dim}(\overline{Y}_s) \geq \text{dim}(Y) = \text{dim}(X_{\eta}) = \text{dim}(X_s).\]

As $\overline{Y}_s\subset X_s$, we have $\text{dim}(X_s) \geq \text{dim}(\overline{Y}_s)$. By 13.1.5 of \cite{egaiv_iii}, fibre dimension is upper semicontinuous for proper morphisms. Hence $\text{dim} (\overline{Y}_s)\geq \text{dim}(Y)$. Furthermore, $Y$ is an irreducible component of $X_{\eta}$ and so $\text{dim}(Y) = \text{dim}(X_{\eta})$ by  the assumption the fibres have pure dimension. By flatness $\text{dim}(X_{\eta}) = \text{dim}(X_s)$. Thus $\text{dim}(\overline{Y}_s) = \text{dim}(X_s)$.

By Lemma 13.1.1 of \cite{egaiv_iii}, a dominant morphism locally of finite type between irreducible schemes is such that the dimensions of the irreducible components of fibres is greater than or equal to the dimension of the generic fibre. Because $\text{dim}(\overline{Y}_s) = \text{dim}(Y)$, we conclude that  $\overline{Y}_s$ is equidimensional. In particular, the underlying topological space of $\overline{Y}_s$ is a union of irreducible components of $X_s$, and as $X_s$ is reduced we find $\overline{Y}_s$ is a union of irreducible components of $X_s$.

As $f$ is flat and locally of finite presentation, the image of any non-empty open set in $X$ is open in $S$, and so contains $\eta$. Thus $\overline{X}_{\eta} = X$, implying the second part of the lemma.
\end{proof}


\begin{lemma}\label{nosharedirred}
Let $f: X\rightarrow S$ be as in \hyperref[compsclosedfromgeneric]{Lemma \ref*{compsclosedfromgeneric}}. Suppose in addition that the smooth locus of each fibre is dense in the fibre. If $\eta_1 \neq \eta_2$ are generic points of distinct irreducible components of $X_{\eta}$, then $(\overline{\eta_1})_s$ and $(\overline{\eta_2})_s$ share no irreducible component of $X_s$ in common.
\end{lemma}

\begin{proof}
Suppose that an irreducible component $Z$ of $X_s$ is contained in $\overline{\eta_1}$ and $\overline{\eta_2}$. Then the local ring $\mathcal{O}_{X,x}$ in $X$ of every smooth point $x\in Z$ over $s$ contains at least two minimal prime ideals, corresponding to the intersection of irreducible components of $X$. Hence $x$ is not a regular point as $\mathcal{O}_{X,x}$ is not an integral domain. This contradicts the assumption that the smooth locus is dense in the fibres. 
\end{proof}

We now focus on the situation where $f:X\rightarrow S$ is a semistable morphism, with $S$ the spectrum of a Noetherian strictly henselian local ring. The closed point of $S$ will be denoted by $s$ and we shall use $\eta$ to denote a generalisation of $s$, which is no longer required to be the generic point of $S$.

\begin{lemma}\label{smnonempty}
Let $Y$ be an irreducible component of $X_{\eta}$. Then $Y$ is geometrically irreducible.
\end{lemma}
\begin{proof}
To show $Y/ \eta$ is geometrically integral it suffices by [\cite{stacks}, Tag 04QM] to show it has a $k$-rational point with $Y/ \eta$ smooth at that point, where $k$ is the residue field of $\eta$. Thus it suffices to show that $\overline{Y}/ \overline{\eta}$ has a section through its smooth locus, where $\overline{Y}$ is the closure of $Y$ in $X$ with the reduced induced structure and $\overline{\eta}$ is the closure of $\eta$ in $S$. 

Define $U = X\setminus \cup_{Y'\neq Y} \overline{Y'}$, where $Y'$ varies over the irreducible components of $X_{\eta}$. Note that $U_s$ is a dense and non-empty open subscheme in each fibre $\overline{Y}_s$ of $\overline{Y}_s$ by \hyperref[nosharedirred]{Lemma \ref*{nosharedirred}}.

Furthermore, the smooth locus of $U/S$ is open and dense in $U$ by Remark \ref{smoothlocusisdense}. Thus $\overline{Y}_s$ admits a non-empty open subscheme of smooth points of $X/S$ that is dense in each irreducible component of $\overline{Y}_s$, and as $X_s$ has sections through the smooth locus of its irreducible components, so too does $\overline{Y}_s$. By 6.2.13 of \cite{liu} this section lifts to a section of $\overline{\eta}$ through the smooth locus of $\overline{Y}$.
\end{proof}

\begin{definition}
The \textit{specialisation map on irreducible components}, $\text{sp}_{s\rightarrow \eta}$, from the irreducible components of $X_s$ to those of $X_{\eta}$ is defined by sending an irreducible component $Z$ of $X_s$ to the unique irreducible component $Y$ of $X_{\eta}$ such that $Z\subset \overline{Y}_s$.
\end{definition}

\begin{remark}\label{transitivityofsp}
The specialisation maps are transitive. If $\eta$ is a generalisation of $s$ and $\zeta$ is a generalisation of $\eta$, then $\text{sp}_{s\rightarrow \zeta} = \text{sp}_{\eta\rightarrow \zeta}\circ \text{sp}_{s\rightarrow \eta}$. This follows by construction: if $Z\subset X_s$, $Y\subset X_{\eta}$, and $W\subset X_{\zeta}$ are irreducible components of their respective fibres with $Y\subset \overline{W}_{\eta}$ and $Z\subset \overline{Y}_s$, then $Z\subset \overline{W}_s$.
\end{remark}

\begin{remark}
\hyperref[smnonempty]{Lemma \ref*{smnonempty}} shows that the irreducible components of $X_{\eta}$ correspond bijectively to those of $X_{\overline{\eta}}$. This will allow us to define a map on vertices of $\Gamma_s$ to those of $\Gamma_{\overline{\eta}}$. To define what the map will be on the edges we must establish some additional results.
\end{remark}

\subsection{The specialisation map on edges}

\begin{lemma}\label{componentscollide}
Suppose that $S = \text{Spec}(R)$ is the spectrum of a Noetherian strictly henselian local ring with closed point $s$, and let $f: X\rightarrow S$ be a semistable morphism. Let $X_1$ and $X_2$ be two irreducible components of $X_s$. Suppose $X_1$ and $X_2$ intersect and contain some element $Z$ of the $k$-strata of $X_s$, $k\geq 1$, such that the label $(b)$ of $Z$ is generated by an element that becomes a unit over $\eta$. Then $\text{sp}_{s\rightarrow \eta}(X_1) = \text{sp}_{s\rightarrow \eta}(X_2)$.
\end{lemma}
\begin{proof}
Fix a chain of prime ideals $m_s \supset p_1 \supset \ldots \supset p_i \supset p_{i+1} \supset m_{\eta}$ of maximum length in $R$, where $m_s$ and $m_{\eta}$ correspond to the points $s$ and $\eta$ respectively. Let $j$ be such that $b\in p_j$ but $b\not\in p_{j+1}$. It suffices to show that $\text{sp}_{s\rightarrow p_{j+1}}(X_1) = \text{sp}_{s\rightarrow p_{j+1}}(X_2)$, for when this holds Remark \ref{transitivityofsp} implies that $\text{sp}_{s\rightarrow \eta}(X_1) = \text{sp}_{s\rightarrow \eta}(X_2)$ also holds.

The properties of being Noetherian and henselian local are preserved under localisation and quotients, and so we may assume that $s = p_j$, $\eta = p_{j+1}$, and that $\text{dim}(R)=1$. By the theorem of Krull-Akizuki (see [\cite{stacks}, Tag 00P7]), the integral closure of $R$ in its field of fractions is a normal noetherian ring of dimension $1$. Localising at a prime ideal lying over the point $s$, we obtain a discrete valuation ring denoted by $F$.

Fix a closed point $x\in Z$. Let $X_F$ denote the base change of $X$ with $\text{Spec}(F)$, and let $x'\in X_F$ be any closed point lying over $x\in Z$. 

As $X_1$ and $X_2$ are geometrically integral and the generic fibre of $X_F$ is isomorphic to that of $X$, it suffices to show that the $\text{sp}_{s'\rightarrow \eta'}(X_{i,F})$ agree, where $s'$ and $\eta'$ are the closed and generic points of $\text{Spec}(F)$, respectively, and $X_{i,F}$ is the component of $X_{F, s'}$ lying over $X_i$. Hence we reduce to the case where $S = \text{Spec}(R)$ is the spectrum of a discrete valuation ring. 

Let $Y_i = \text{sp}_{s\rightarrow \eta}(X_i)$. Let $U \rightarrow \text{Spec}(R'[x_1, \ldots, x_n]/(x_1 \cdots x_l - b))$ be a local chart of $x\in Z$, with $g:U \rightarrow X$ an \'{e}tale neighbourhood of $x$, and $R'$ an \'{e}tale ring over $R$. By restricting to an open subscheme of $U$, we may assume that $U$ is connected. As $b$ is a unit over $\eta$ we find that the generic fibre of $U$ is smooth. 

Because $g(U)$ is an open set containing the generic points of $X_1$ and $X_2$, so too must it contain the generic points of $Y_1$ and $Y_2$. Furthermore, $g(U)$ is connected, being the continuous image of a connected space, and flat over $S$ as it is open in $X$. The \'{e}tale local rings points of $X_s$ are reduced by Lemma 15.42.4 of [\cite{stacks}, Tag 07QL], and the same lemma implies that $X_s$ is reduced. 


Hence $g(U)$ is connected, flat over $S$, and has reduced special fibre, so by Lemma 36.26.7 of [\cite{stacks}, Tag 055C] the generic fibre of $g(U)$ over $S$ is connected, and as it is smooth it is also irreducible. Hence the generic points of $Y_1$ and $Y_2$ are equivalent, whence $\text{sp}_{s\rightarrow \eta}(X_1) =\text{sp}_{s\rightarrow \eta}(X_2)$.
\end{proof}

\begin{definition}
Let $f: X \rightarrow S$ be a finite-type morphism. Then $\text{Sing}_f(X)$ is the non-smooth locus of $f$, viewed as a closed subscheme of $X$ with the reduced induced structure. 
\end{definition}

As $f: X\rightarrow S$ is flat and of finite presentation, $f$ is smooth at a point $x$ if and only if $x$ is a smooth point of the fibre $X_{f(x)}$. By Lemma 3.5 of \cite{li}, the set of connected components of the non-smooth locus $\text{Sing}_f(X)$ of $X/S$ is in bijection with the set of connected components of the singular locus of $X_s$ via taking the fibre over the closed point. Let $C$ be a connected component of $\text{Sing}_f(X)$, and let $\eta$ be any generalisation of $s$ with $C_{\eta}$ non-empty. Lemma 3.15a of \cite{li} shows that the label of $C_{\eta}$ is the image of the label of $C_s$ in $\mathcal{O}_{S,\overline{\eta}}$, where $\overline{\eta}$ corresponds to a choice of an algebraic closure of $k(\eta)$. In particular, if the label of $C_s$ is $0$, then the label of any non-empty fibre $C_{\eta}$ of $C$ is $0$.

\begin{remark}\label{fittingideal}
By Lemma 30.10.3 of [\cite{stacks}, Tag 0C3H], the non-smooth locus of a flat morphism locally of finite presentation with equidimensional fibres of dimension $d$ is cut out by the $d$-th fitting ideal of $\Omega_{X/S}$. Furthermore, the fitting ideal commutes with arbitrary base change by Lemma 30.10.1 of [\cite{stacks}, Tag 0C3H]. 

In particular, if $X/S$ is a semistable morphism with $n$-dimensional fibres then $\text{Sing}_f(X)$ is cut out by the $n$-th fitting ideal. If \linebreak $U\rightarrow T = \text{Spec}(R[x_1, \ldots, x_n]/(x_1\cdots x_l - b))$ is a local chart of some closed point $x\in X$, then $\text{Sing}(U)\cong \text{Sing}(X)\times_X U$ and also $\text{Sing}(U)\cong \text{Sing}_(T)\times_T U$.
\end{remark}

\begin{lemma}\label{flatnessofsing}
Let $X/S$ be a semistable morphism with $S$ a Noetherian strictly henselian local scheme with closed point $s$, and let $C$ a connected component of the singular locus $\text{Sing}_f(X)$ of $X$ whose label is $0$. Then $C/S$ is flat and proper over $S$.



\end{lemma}
\begin{proof}
$X$ is proper over $S$ and $C$ is a closed subscheme of $X$, so $C$ is proper over $S$.

Any point $x\in C$ admits a local chart \linebreak $U \rightarrow T = \text{Spec}(R[x_1,\ldots, x_n]/(x_1\cdots x_l))$, where $g: U \rightarrow X$ is an \'{e}tale neighbourhood of $x$. As $g$ and $h$ are open and flatness can be checked \'{e}tale locally, we may assume that $g$ and $h$ are surjective by replacing $X$ and $T$ with their respective images under $g$ and $h$.

The non-smooth locus $\text{Sing}(T)$ of $T$ is defined by the ideal $I$ generated by polynomials of the form $x_1\cdots \hat{x_i}\cdots x_l$, $1\leq i \leq l$, where $\hat{x_i}$ denotes the exclusion of $x_i$ from the term. Viewed as a module over $R$, $R[x_1,\ldots, x_n]/I$ is free with basis \[\{x_1^{i_1}\cdots x_n^{i_l}| i_j\in \mathbb{Z}\, \forall \, i, j \text{ and there exists at least two integers }k\in [1,l] \text{ with } i_k=0 \}.\] In particular, $\text{Sing}(T)$ is flat over $S$.

By \hyperref[fittingideal]{Remark \ref*{fittingideal}}, $\text{Sing}(T)$ (resp. $\text{Sing}(U)$ and $\text{Sing}_f(X)$) is cut out by the $d$-th fitting ideal of the module of relative differentials of $T$ (resp. $U$ and $X$) over $S$, where $d = n-1$. As formation of the fitting ideal commutes with arbitrary base change, we have that $\text{Sing}(U)$ is isomorphic to $\text{Sing}(T)\times_T U$, which is \'{e}tale over $\text{Sing}(T)$ and hence flat over $S$. Finally, as $\text{Sing}(U) \cong \text{Sing}_f(X)\times_X U$ is flat over $S$ and flatness can be checked \'{e}tale locally, we conclude that $\text{Sing}_f(X)$ is flat over $s$. 



\end{proof}

\begin{corollary}\label{edgesgi}
Let $X/S$ be a semistable morphism, with $(S,s)$ the spectrum of an integral Noetherian strictly henselian local ring with generic point $\eta$. Let $Y\subset X_{\eta}$ denote an irreducible component of the non-smooth locus of $X_{\eta}$. Then $Y$ is geometrically irreducible.
\end{corollary}
\begin{proof}
The label of the connected component of $\text{Sing}_f(X)$ containing $Y$ is necessarily $0$ by the integrality of $S$, and so by \hyperref[flatnessofsing]{Lemma \ref*{flatnessofsing}} we conclude $C$ is flat over $S$. The remainder of the proof proceeds in a manner completely analogous to that of \hyperref[smnonempty]{Lemma \ref*{smnonempty}}.

\end{proof}

Let $C$ be a connected component of the non-smooth locus of $X_s$ with label $(b)$, and suppose the generator $b$ of $(b)$ is not a unit over $\eta$ for some generalisation $\eta$ of $s$. After taking a suitable quotient in $S$ we may assume that $\eta$ is the generic point and that $b=0$. We now describe an analogue for $\text{sp}_{s\rightarrow \eta}$ to the $1$-strata of $X_s$ contained in $C$. 

Let $Z\subset C$ be a $1$-strata with generic point $\eta_Z$, and let \linebreak $U\rightarrow T = \text{Spec}(R[x_1,\ldots, x_n]/(x_1 x_2)$ denote a local chart of a closed point $x\in Z$. Fix a lift $p$ of $\eta_Z$ to $U$, necessarily mapping to the point defined by $(m_s, x_1, x_2)$ of $T$. As $(m_s, x_1, x_2)$ generalises in $T$ to the point $(m_{\eta}, x_1, x_2)$, we can find a point $q\in U$ that is a generalisation of $p$ lying over the point defined by $(m_{\eta}, x_1, x_2)$. 

Let $\eta_Y$ be the image of $q$ in $X$. This is a generalisation of $\eta_Z$, and as $U$ is \'{e}tale over $T$ and $X$ with the generic points of the $1$-strata of the fibres of $U$ lying over the generic points of the $1$-strata of the fibres of $T$ and $X$, we see that $\eta_Y$ is the generic point of an element $Y$ of the $1$-strata of $X_{\eta}$. 

\begin{lemma}
There is a unique point $\eta_Y\in X_{\eta}$ generalising $\eta_Z$ that is the generic point of an element of the $1$-strata of $X_{\eta}$. 
\end{lemma}
\begin{proof}
The above construction gives existence. Uniqueness follows from the fact the connected component $C$ of $\text{Sing}_f(X)$ is flat over $S$ by \hyperref[flatnessofsing]{Lemma \ref*{flatnessofsing}}, and then arguing as in the proof of \hyperref[smnonempty]{Lemma \ref*{smnonempty}}.

\end{proof}

\begin{definition}
Let $\eta$ be a generalisation of the closed point $s$. Let $Z$ be a $1$-strata of $X_s$ and assume that the label of $Z$ is not a unit over $\eta$. We set $\text{sp}'_{s\rightarrow \eta}(Z) = Y$, where $Y$ is the unique $1$-strata of $X_{\eta}$ as constructed above.
\end{definition}




\subsection{The specialisation map}

We now define the specialisation map from $\Gamma_s$ to $\Gamma_{\overline{\eta}}$. By \hyperref[smnonempty]{Lemma \ref*{smnonempty}} and \hyperref[edgesgi]{Corollary \ref*{edgesgi}} we find that $\Gamma_{\overline{\eta}}\cong \Gamma_{\eta}$, and so it suffices to define the specialisation map from $\Gamma_s$ to $\Gamma_{\eta}$. 

We define $\text{sp}: \Gamma_s \rightarrow \Gamma_{\eta}$ as follows: A vertex $v$ of $\Gamma_s$ corresponds to an irreducible component $Y$ of $X_s$. Let $w$ be the vertex corresponding to $\text{sp}_{s\rightarrow \eta}(Y)$ and set $\text{sp}(v) = w$. If $e$ is an edge between vertices $v$ and $w$ of $\Gamma_s$ corresponding to an element $Z$ of the $1$-strata with label $b$ we have two possibilities: 
\begin{itemize}
\item $b$ is a unit over $\eta$. In this case we have $v$ and $w$ map to the same vertex in $\Gamma_{\eta}$ by \hyperref[componentscollide]{Lemma \ref*{componentscollide}}, and we define $\text{sp}(e)$ to be the vertex $v$ and $w$ map to. 
\item $b$ is not a unit over $\eta$. In this case we set $\text{sp}(e) = f$, where $f$ is associated to the element $Y$ of the $1$-strata of $X_{\eta}$ such that $\text{sp}'_{s\rightarrow \eta}(Z) = Y$. As generalisations are transitive we have that $f$ is incident with $\text{sp}(v)$ and $\text{sp}(w)$.
\end{itemize}
The map of the edge-labellings sends a label $(b)$ to the principal ideal $(b)$ of $\mathcal{O}_{X, \overline{\eta}}$ when $\text{sp}$ sends an edge to an edge; otherwise no label is required. 

We conclude with a lemma that will be used in the proof of \hyperref[weiliscartier]{Lemma \ref*{weiliscartier}}.

\begin{lemma}\label{injectivityofsp}
Suppose that $S$ is an excellent integral strictly Henselian local ring with closed point $s$ and generic point $\eta$. Let $f: X\rightarrow S$ be a semistable morphism. Suppose we have two irreducible components $Z_1$ and $Z_2$ of $X_s$ that share a common $1$-strata whose label is $0$. Then $\text{sp}_{s\rightarrow \eta}(Z_1)\neq \text{sp}_{s\rightarrow \eta}(Z_2)$.
\end{lemma}


\begin{proof}
Let us proceed by contradiction and assume $\text{sp}_{s\rightarrow \eta}(Z_1) = \text{sp}_{s\rightarrow \eta}(Z_2) = Y$, and let $\overline{Y}$ denote its closure in $X$. We shall show that the normalisation of $\overline{Y}$ is smooth with irreducible special fibre, and hence that $Z_1 = Z_2$, leading to the desired contradiction.

Consider any point $x$ in the special fibre $\overline{Y}_s$ of $\overline{Y}$ lying in a $1$-strata. Fix a local chart $U\rightarrow T = \text{Spec}(R[x_1, \ldots, x_n]/(x_1 x_2))$ of $x$.

The normalisation $T^{\nu}$ of $T$ is the disjoint union of its irreducible components, which are themselves defined by ideals of the form $(x_i)$ for $1\leq i \leq 2$.

As normalisation commutes with \'{e}tale base change by [\cite{stacks}, Tag 082F], we have \[U^{\nu} = U \times_{T} T^{\nu}.\] Hence $U^{\nu}$ is \'{e}tale over $T^{\nu}$, and as $T^{\nu}$ is smooth over $S$, so too is $U^{\nu}$.

Again using the fact that normalisation commutes with \'{e}tale base change, we note that $U^{\nu}$ is \'{e}tale over $X^{\nu}$. Thus as smoothness is \'{e}tale local on the source we conclude that $(\overline{Y})^{\nu}$ is smooth over $S$. 

It remains to show that $(\overline{Y})^{\nu}$ is proper over $S$. As $\overline{Y}$ is excellent (being of finite type over $S$) this follows by Theorem 8.2.39 of \cite{liu}.

Hence $(\overline{Y})^{\nu}$ is proper and smooth over $S$, and so by Corollary 15.5.4 of \cite{egaiv_iii} its special fibre is connected and smooth, hence irreducible. The morphism $(\overline{Y})^{\nu} \rightarrow Y$ is surjective by Lemma 28.51.5 of [\cite{stacks}, Tag 035E], and so the generic point of the special fibre of $(\overline{Y})^{\nu}$ must map to a generic point of one of the irreducible components of $Y_s$, say $\eta_{Z_1}$. As $\eta_{Z_2}$ is also in the image, it is necessarily a specialisation of $\eta_{Z_1}$, contradicting the assumption $Z_1\neq Z_2$. 
\end{proof}


\section{Cartier divisors and alignment}\label{sec:four}
\subsection{Constructing Cartier divisors on $X$}

Having defined the sp map we are ready to construct Cartier divisors on $X$ in a manner analogous to the construction in chapter 5 of \cite{holmes}. In this subsection we shall consider semistable morphisms $X\rightarrow S$ where $S=\text{Spec}(R)$ is the spectrum of an excellent regular strictly henselian local ring. We shall denote the closed point of $S$ by $s$ and let $\Gamma_s$ denote the dual graph. The following definition mirrors that of Definition 5.5 of \cite{holmes}.

\begin{definition}
Let $a\in R$ be non-zero and a non-unit. Let $V(a)$ denote the set of edges of $\Gamma_s$ whose labels $b$ are such that $(b)^m=(a)$ as ideals in $R$ for some $m\geq 1$. Let $\Gamma_s(a)$ denote the graph obtained from $\Gamma_s$ by removing all edges in set $V(a)$.
\end{definition}

Let $\eta_1, \ldots, \eta_r$ denote the generic points of $\text{Spec}(R/a)$, and let $m_i$ denote the order of vanishing of $a$ at $\eta_i$. Let $H$ denote the set of vertices of a connected component of $\Gamma_s(a)$. For every $1\leq i \leq r$, let $Z_i^1, \cdots, Z_i^{s_i}$ denote the vertices of $\Gamma_{\eta_i}$ which are images under sp of vertices in $H$. Note that each $Z_i^j$ can be viewed as a prime Weil divisor on $X$.

We now construct a Weil divisor on $X$ analogous to that of Definition 5.6 of \cite{holmes}.

\begin{definition}
Define a Weil divisor $\text{div}(a;H)$ by

\[\text{div}(a;H) = \sum_{i=1}^{r} m_i \sum_{j=1}^{s_i} Z_i^j.\]
\end{definition}

\begin{lemma}\label{weiliscartier}
The Weil divisor $\text{div}(a;H)$ defined above is a Cartier divisor on $X$.
\end{lemma}

\begin{proof}
Our proof shall closely follow that of Lemma 5.7 of \cite{holmes}. It suffices to check $D$ is Cartier at closed points of the closed fibre by Lemma 5.8 of \cite{holmes}, and in particular at points in the non-smooth locus. 

Let $x$ be a closed point in the non-smooth locus of $X_s$, and consider a local chart \[U \rightarrow \text{Spec}(R[x_1,\ldots, x_n]/(x_1\cdots x_l - b)),\] where $U$ is an \'{e}tale neighbourhood of $x$ and where necessarily $b\in m_s$ for $m_s$ the maximal ideal of $R$. From the description of the local chart we see that $x$ belongs to a unique $(l-1)$-strata $C$. Hence there are at most $l$ vertices of $\Gamma_s$ whose corresponding irreducible components of $X_s$ contain $C$.

Suppose that $k$ of these vertices, $v_1, \ldots, v_k$ are in $H$, and $r$ are not in $H$, say $v_{k+1}, \ldots, v_{k+r}$. If $k=0$ we have that $\text{div}(a;H)$ is locally cut out by a unit. If $r=0$ one shows as in \cite{holmes} shows that $\text{div}(a;H)$ is cut out locally by $a$. So assume now that $k$ and $r$ are greater than $0$.

Let $g: (U,x')\rightarrow (X,x)$ an \'{e}tale neighbourhood of $x$ with $h: U\rightarrow T = \text{Spec}(R[x_1,\ldots, x_n]/(x_1\cdots x_l - b)$ a local chart. It suffices to show that $g^*\text{div}(a;H)$ is a Cartier divisor in a neighbourhood of $x'$ by Lemma 2.9 of \cite{holmes}. 

Base changing along $V(x_i, m_s)\rightarrow s$ for any $1\leq i \leq l$, we find that $V(x_i, m_s)\times_T U$ is smooth over $s$. In particular there exists a unique generalisation $x_i'$ of $x'$ in $U_s$ lying over each point $(x_i, m_s)$ for all $1\leq i\leq l$, as otherwise any point lying over $x'$ in $V(x_i, m_s)\times_T U$ would not be smooth over $s$. 

Each irreducible component $V(x_i, m_s)$ of $T_s$, $1\leq i\leq l$, corresponds to one of the $k+r$ vertices of $\Gamma_s$ containing the point $x$, though in general $l\geq k+r$ and so the correspondence may not be injective. Let $\alpha\geq k$ be such that $V(x_i, m_s)\subset T_s$ corresponds to a vertex in $H$ for $1\leq i \leq \alpha$, and for $\alpha+1\leq j\leq l$ the irreducible components $V(x_j, m_s)$ correspond to vertices not in $H$, where $l-\alpha \geq r$. 

Fix an irreducible component of $R/(a)$, say $\eta_1$, and let $\text{sp}: \Gamma_s\rightarrow \Gamma_{\eta_1}$ be the specialisation map. Note that $\text{sp}$ does not map any of the vertices in $H$ to the image of a vertex not in $H$, as otherwise there exists an edge whose label is a unit between $H$ and the complement of $H$ by \hyperref[componentscollide]{Lemma \ref*{componentscollide}} and \hyperref[injectivityofsp]{Lemma \ref*{injectivityofsp}}, contradicting how $H$ was chosen. 

If $V_i = \text{sp}(v_i)$, then $\cup_{1\leq i \leq k} \overline{V}_i \supset \cup_{1\leq i \leq k} v_i$, and the same holds along pullbacks via $g$. Similarly, $\cup_{k+1\leq j \leq k+r} \overline{V}_j \supset \cup_{k+1\leq j \leq k+r} v_j$.

In particular $h^*(x_1\cdots x_{\alpha})$ does not vanish on $\cup_{k+1\leq j \leq k+r} g^*\overline{V_j}$, nor does $h^*(x_{\alpha+1}\cdots x_l)$ vanish on $\cup_{1\leq i \leq k} g^*\overline{V_i}$. Hence, as $\text{div}a = \text{div}(x_1\ldots x_{\alpha})^n + \text{div}(x_{\alpha+1}\ldots x_l)^n$, we conclude that $h^*(x_1\cdots x_{\alpha})$ vanishes on $g^*\overline{V_i}$, $1\leq i \leq k+1$, with the same multiplicity as $a$. This concludes the proof.

\end{proof}

\subsection{Cartier functions on graphs}

Suppose now that $X\rightarrow S$ is smooth over a dense open set $U\subset S$. In order to introduce the notion of Cartier divisors on the dual graph $\Gamma_s$ of $X_s$ we first recall the definition and existence of "test curves" in $S$ from section 5.1 of \cite{holmes}. 

\begin{definition}[Definition 5.1 of \cite{holmes}]
Let $S$ be a scheme, $s\in S$, and $U\subset S$ an open subscheme. A \textit{non-degenerate trait} in $S$ through $s$ is a morphism $\phi: \mathcal{T} \rightarrow S$ where $\mathcal{T}$ is the spectrum of a discrete valuation ring and $\phi$ maps the closed point of $\mathcal{T}$ to $s$ and the generic point of $\mathcal{T}$ to a point in $U$.
\end{definition}

\begin{lemma}[Lemma 5.2 of \cite{holmes}]\label{traitsexist}
Let $S$ be a Noetherian scheme, $s\in S$ and $U\subset S$ a dense open subscheme. Then there exists a non-degenerate trait $X$ in $S$ through $s$.
\end{lemma}

With the existence of non-degenerate traits, we now define Cartier functions on $\Gamma_s$. Recall that we have an edge-labeling $l: E \rightarrow M_s$ on $\Gamma_s$, where $M_s$ denotes the monoid of principal ideals of $\mathcal{O}_{S,s}$.

\begin{definition}[c.f. Definition 5.3 of \cite{holmes}]
Let $(S,s)$ be the spectrum of a Noetherian strictly henselian local ring and let $f:X\rightarrow S$ be a semistable morphism that is smooth over a dense open set $U\subset S$. Denote by $l:E\rightarrow M_s$ the edge-labelling of $\Gamma_s$ by elements of $M_s$ as in \hyperref[edgelabelling]{Definition \ref*{edgelabelling}}. 

Let $\phi: \mathcal{T} \rightarrow S$ be a non-degenerate trait through $s$, and let $\text{ord}$ denote the standard valuation on elements of $\Gamma(X)$. We say that a vertex-labelling $m:V\rightarrow \mathbb{Z}$ of $\Gamma_s$ is \textit{$T$-Cartier} if for every edge $e\in \Gamma_s$ incident with vertices $v_1$ and $v_2$ we have that $m(v_1)-m(v_2)$ is divisible by $\text{ord}(\phi^*(l(e)))$, where $l(e)$ is the label of edge $e$.

Given a fibral Weil divisor $D$ on $X_{\mathcal{T}}$, we define a vertex labelling $m$ of $\Gamma_s$ by attaching to $v$ the multiplicity of $D$ along the irreducible component of $X_s$ corresponding to $v$. 
\end{definition}

The significance of this definition is found in the following lemma:

\begin{lemma}\label{cartieriffTcartier}
With the notation as above, a vertex labelling $m$ is $\mathcal{T}$-Cartier if and only if there exists a fibral Cartier divisor $D$ on $X_{\mathcal{T}}$ whose labelling is $m$.
\end{lemma}
\begin{proof}
Let $\mathcal{T} = \text{Spec}(A)$ have uniformiser $\pi$ and standard valuation $\text{ord}$ on elements of $A$. 

Assume first that we have a vertical Cartier divisor $D$ on $X_{\mathcal{T}}$ with associated vertex labelling $m$. Let $e$ be an edge of $\Gamma_s$, incident with vertices $v_1$ and $v_2$ and with label $(b)\subset \mathcal{O}_{S,s}$, and let $x\in X_s$ be a closed point in the $1$-strata associated to $e$. By the existence of local charts we see that the $k(s)$-rational points of the $1$-strata are dense in the $1$-strata, and so we may assume that $x$ is $k(s)$-rational. 

Let $U\rightarrow \text{Spec}(R[x_1,\ldots, x_n]/(x_1 x_2 - b)$ be a local chart of $x$. Via base change over $S$ we obtain a local chart $U'\rightarrow \text{Spec}(A[x_1,\ldots, x_n]/(x_1x_2 - \pi^f)$ of an $(A/\pi)$-rational point $x'\in X_{\mathcal{T}}$ lying over $x\in X$. Here $f = \text{ord}\phi^*b$. 

In particular, as $x'$ is a rational point in the special fibre, the completion of the \'{e}tale local ring of $x'$ in $X_{\mathcal{T}}$ is isomorphic to $\hat{A}[[x_1, \ldots, x_n]]/(x_1x_2 - \pi^f)$. As $\text{Spec}(\hat{A}[[x_1, \ldots, x_n]]/(x_1x_2 - \pi^f))$ is flat over $X_{\mathcal{T}}$, $D$ pulls back to a vertical Cartier divisor on $\text{Spec}(\hat{A}[[x_1, \ldots, x_n]]/(x_1x_2 - \pi^f))$. The irreducible components of the special fibre of $\text{Spec}(\hat{A}[[x_1, \ldots, x_n]]/(x_1x_2 - \pi^f))$ are defined by ideals of the form $(x_i, \pi)$ for $1\leq i\leq 2$, and the function $x_i$ vanishes on $(x_i, \pi)$ to order $f$. After possibly reordering the $x_i$'s, we may assume that $(x_1,\pi)$ lies over $v_1$ and $(x_2, \pi)$ lies over $v_2$. From this we see that $f | m(v_1) - m(v_2)$.

Conversely, let $m$ be a $\mathcal{T}$-Cartier vertex labelling. Let $D$ be the associated Weil divisor on $X_\mathcal{T}$. By adding or subtracting a multiple of $\text{div}\pi$ we may assume that $D$ is effective and that at least one of the $v_i$'s has coefficient 0. By Lemma 5.8 of \cite{holmes}, to show $D$ is Cartier it suffices to show it is Cartier at closed points in the closed fibre, and in particular at points in the non-smooth locus. Let $x$ be such a point, and fix a local chart $h: U\rightarrow \text{Spec}(A[x_1,\ldots, x_n]/(x_1\cdots x_l - \pi^f)$, where $g:(U,x')\rightarrow (X,x)$ is an \'{e}tale neighbourhood of $x$. By Lemma 2.9 of \cite{holmes}, to show $D$ is Cartier at $x$ it suffices to show $g^*D$ is Cartier at $x'$. 

Let $g^*D = \sum_i m_i w_i$, where the $w_i$'s are prime Weil divisors supported in the special fibre of $U$. We may assume each $w_i$ contains $x'$ by choosing a smaller neighbourhood of $x'$ as needed. 

As $h$ is \'{e}tale, each $w_i$ lies over some prime Weil divisor in the special fibre of $\text{Spec}(A[x_1,\ldots, x_n]/(x_1\cdots x_l - \pi^f)$. The irreducible components of the special fibre of $\text{Spec}(A[x_1,\ldots, x_n]/(x_1\cdots x_l - \pi^f)$ are smooth over $\text{Spec}(A/\pi)$, and so $h(w_i)\neq h(w_j)$ for $i\neq j$, and after rearranging we may write the image of $h(w_i)$ as the point defined by $(\pi, x_i)$ for $1\leq i \leq l$. Then the rational function $\Pi_{1\leq i \leq l} x_i^{m_i}$ pulls back to the divisor $g^*D$ on $U$, whence $g^*D$ and hence $D$ are vertical Cartier divisors.

\end{proof}

\subsection{Alignment and Cartier divisors}

Recall that $f:X\rightarrow S$ is said to be aligned at a geometric point $\overline{s}\rightarrow S$ if the dual graph $\Gamma_{\overline{s}}$ of $X_{\overline{s}}$ is aligned with respect to the edge-labelling $l:E\rightarrow M_{\overline{s}}$, as in Definition \ref{defofalignment}. The following two lemmas will be used in proving \hyperref[mainresult1]{Theorem \ref*{mainresult1}}. They imply the existence and non-existence of certain Cartier divisors on $X$ when $X/S$ is aligned and not aligned, respectively. 

\begin{lemma}[c.f. Lemma 5.12 of \cite{holmes}]\label{forwardimplication}
Let $(S,s)$ be the spectrum of an excellent regular strictly henselian local ring, and $f:X\rightarrow S$ a semistable morphism that is smooth over a dense open subscheme $U\subset S$. Suppose that $X/S$ is aligned at $s$, and let $\phi: \mathcal{T} \rightarrow S$ denote a non-degenerate trait through $s$. Let $m$ denote a $\mathcal{T}$-Cartier vertex labelling of $\Gamma_s$ that takes the value $0$ on some fixed vertex $v_0$. Then there exists a Cartier divisor $D$ on $X/S$, trivial over the generic point of $S$, such that $m$ is the vertex labelling corresponding to $\phi^*D$.
\end{lemma}
Our proof will closely follow that of Lemma 5.10 and Lemma 5.12 of \cite{holmes}.
\begin{proof}
We begin by defining three labellings on the edges of $\Gamma_s$. The first labelling, denoted $l_{orig}$, is the usual labelling by the monoid $M_s$ of principal ideals of $\mathcal{O}_{S,s}$.

The second labelling is the quotient of $l_{orig}$ by the equivalence relation that $[a]=[b]$ if and only if $(a)^n = (b)^m$ as ideals of $\mathcal{O}_{S,s}$ for some positive integers $n$ and $m$. This labelling is denoted by $l_Q$, where $Q$ denotes the quotient of $M_s$ this equivalence relation.

The final labelling on the edges is determined by sending a label in $M_s$ to $\text{ord} \phi^*(l(e)) \in \mathbb{N}\cap \{\infty\}$, where $\text{ord}$ denotes the usual valuation on $\Gamma(\mathcal{T})$. This labelling is denoted by $l_{\mathcal{T}}$. By \hyperref[cartieriffTcartier]{Lemma \ref*{cartieriffTcartier}}, a vertex labelling is $\mathcal{T}$-Cartier with respect to $l_{\mathcal{T}}$ if and only if it arises from a vertical Cartier divisor on $X_{\mathcal{T}}$.

As $l_Q$ is constant on circuits by the assumption $X/S$ is aligned at $s$, Lemma 5.11 of \cite{holmes} implies that we may write the $\mathcal{T}$-Cartier vertex labelling $m$ as a finite sum of functions $m_i: V \rightarrow \mathbb{Z}$ such that

\begin{itemize}
\item All of the $m_i$'s take the value $0$ at $v_0$, where $v_0$ is the fixed vertex on which $m$ is $0$;
\item Each $m_i$ is $\mathcal{T}$-Cartier;
\item Each $m_i$ satisfies the following: There exists a subset $B\subset E$ of edges whose labels in $l_Q$ are the same, and such that $m_i$ is constant on connected components of the graph obtained by deleting all edges in set $B$.
\end{itemize}

Hence we may reduce to the case where $m$ satisfies the above three properties. In particular there exists an element $h\in Q$ such that $m$ is constant on connected components of $\Gamma_s$ obtained by deleting every edge with label $h$. By definition of $Q$, there exists an element $a\in \mathcal{O}_{S,s}$ such that for each edge $e$ with $l_Q(e) = h$, we have $(a) = (l_{orig}(e))^{n_e}$ for some positive integer $n_e$ depending on $e$. 

As $S$ is factorial, we may decompose $a$ into a product of prime elements to find that  there exists some element $\alpha \in \mathcal{O}_{S,s}$ such that $l_{orig}(e)$ is a power of $\alpha$ for each $e$ with label $l_Q(e)=h$. Because $m$ is constant on connected components of $\Gamma_s$ obtained by deleting all edges with $l_Q(e)=h$, we have that $\text{ord}\phi^*(\alpha)$ divides $l_{\mathcal{T}}(e)$ for every edge $e$ of $\Gamma_s$. Moreover, as $m$ takes the value $0$ at at least one vertex, we conclude that $\text{ord}\phi^*(\alpha)$ divides $m(v)$ for all $v$. 

Let $H$ be a connected component of $\Gamma_s$ with edges of label $l_Q(e) = h$ removed on which $m$ is non-zero. Then $m(v) = r\text{ord}\phi^*(\alpha)$ on all vertices of $H$ for some $r>0$, and we may assume by the decomposition of $m$ into the $m_i$'s as above that $m$ is $0$ outside of $H$. 

Let $D = \text{div}(\alpha^r,H)$. This is Cartier by \hyperref[weiliscartier]{Lemma \ref*{weiliscartier}}, trivial over the generic point by construction. We shall show that the vertex labelling associated to $\phi^*D$ is $m$, concluding the proof.

To see this, note that $\phi^*D$ is $0$ outside of $H$ by construction. The proof of \hyperref[weiliscartier]{Lemma \ref*{weiliscartier}} showed that $\phi^*D$ is cut out by $\alpha^r$ near the generic points of vertices in $H$, and hence that

\[\text{ord}\phi^*\alpha^r = r\text{ord}\phi^* \alpha = m(v)\]
on vertices of $H$, as required.
\end{proof}

\begin{lemma}[c.f. Lemma 5.13 of \cite{holmes}]\label{reverseimplication}
Let $S$ be an excellent regular strictly Henselian local scheme with closed point $s$, $U\subset S$ a dense open subscheme, and $X\rightarrow S$ a semistable morphism that is smooth over $U$. If $X/S$ is not aligned at some $s\in S$, then for every non-degenerate trait $\phi:\mathcal{T}\rightarrow S$ through $s$ we can find a Cartier divisor $D$ on $X_{\mathcal{T}}$, trivial over the generic point of $\mathcal{T}$, such that there does not exist a Cartier divisor $E$ on $X/S$, trivial over $U$, with $\phi^*E$ linearly equivalent to $D$.
\end{lemma}

\begin{proof}
Let $\Gamma_s$  be the dual graph of $X_s$, and let $\phi: \mathcal{T}\rightarrow S$ be a non-degenerate trait through $s$. By assumption there exists a circuit with vertices $v_0, v_1, \ldots, v_N$, in order, with edge $e_i$ incident with $v_i$ and $v_{i+1}$ (taken modulo N) in $\Gamma_s$, and where if $(b_i)$ is the label of $e_i$, then (without loss of generality) $(b_0)^{k_0} \neq (b_1)^{k_1}$ for all $(k_0, k_1)\in \mathbb{N}^2$.

Let $d$ be the product of all the values of $\text{ord}_T(\phi^*b_l)$ as $b_l$ runs over the distinct edge labels of $\Gamma_s$. Set $D = dv_1$, which is a Cartier divisor on $X_{\mathcal{T}}$ by \hyperref[cartieriffTcartier]{Lemma \ref*{cartieriffTcartier}}. Suppose there exists a Cartier divisor $E$ on $X$ such that $\phi^*E$ is linearly equivalent to $D$. We shall show that this implies the existence of a multiplicative relation between $b_0$ and $b_1$, thereby deriving a contradiction.

By 6.2.13 of \cite{liu}, there exists sections $\sigma_i: S \rightarrow X$ through the smooth locus of $v_i$ for all $v_i$ in the circuit. Here we view $v_i$ as both a vertex of $\Gamma_s$ and the corresponding irreducible component of $X_s$. Let $(\sigma_0)_{\mathcal{T}}$ denote the section from $\mathcal{T}$ to $X_{\mathcal{T}}$ induced by $\sigma_0$. Because $D$ is generically trivial on $X_0$, $(\sigma_0)^*_{\mathcal{T}}D=0$. As $S$ is regular, $\sigma_0^*E$ is a Cartier divisor, and the same is true for $\sigma_i^*E$ for all $i$. We may assume that $\sigma_0^*E=0$ after multiplying by a suitable Cartier divisor.

As $\pi_*\mathcal{O}_X = \mathcal{O}_S$ (by Exercise 9.3.11 of \cite{fantechi}) we then have $\phi^*E=D$, as otherwise $\phi^*E - D$ is a multiple of $\text{div}\pi$ for $\pi$ the uniformzer of $\Gamma(\mathcal{T})$, contradicting the fact that $\sigma_0^*E = 0$ and $(\sigma_0)^*_{\mathcal{T}}D = 0$.

Let $f_i\in \mathcal{O}_{S,s}$ be such that $\text{div}(f_i) = \sigma_i^*E$. Given $a, b \in \mathcal{O}_{S,s}$, we shall write $a\sim b$ if $a$ and $b$ differ by a unit. Thus $f_0\sim 1$. By \hyperref[formofCartier]{Lemma \ref*{formofCartier}}, we have $f_1\sim b_0^{d_0} f_0$ for some $d_0\in \mathbb{Z}$, whence $d_0 = d/\text{ord}_T(b_0)$, and so we can equivalently write $f_1\sim (b_0)^{\frac{d}{\text{ord}_T(b_0)}}f_0$. Similarly, $f_2 \sim (b_1)^{-\frac{d}{\text{ord}_T(b_1)}}f_1$, and $f_2\sim \ldots \sim f_N \sim f_0$. We conclude that $(b_1)^{\frac{d}{\text{ord}_T(b_1)}} = (b_0)^{\frac{d}{\text{ord}_T(b_0)}}$, contradicting the non-alignment of $\Gamma_s$.
\end{proof}

\begin{lemma}\label{formofCartier}
Let $S$ be the spectrum of a regular strictly henselian local ring with closed point $s$, and $X\rightarrow S$ a semistable morphism. Let $v_1, \ldots, v_n$ denote the set of irreducible components of $X_s$, where for each $i$ we have a section $\sigma_i: S \rightarrow X$ passing through the smooth locus of $v_i$. Let $E$ be a Cartier divisor on $X$, trivial over a dense open subscheme $U\subset S$, and let $f_i\in \text{Frac}(\mathcal{O}_{S,s})^*$ be such that $\text{div}f_i = \sigma_i^*E$. 

Suppose there is an edge $e$ incident with $v_1$ and $v_2$ in the dual graph of $X_s$. Let $b\in \mathcal{O}_{S,s}$ be the label of the associated $1$-strata. Then, up to multiplication by units in $\mathcal{O}_{S,s}$, we have $f_1 = b^{\delta}f_2$ for some $\delta \in \mathbb{Z}$.
\end{lemma}

\begin{proof}
Let $p\in X_s$ be a $k$-rational point on $v_1$ and $v_2$ lying in the $1$-strata of $X_s$ associated to $e$, and fix a local chart $U\rightarrow \text{Spec}(R[x,y, x_1, \ldots, x_n]/(xy-b))$, where $(U,x')\rightarrow (X,x)$ is an \'{e}tale neighbourhood of $x$. As $p$ is $k$-rational the image of $x'$ in $\text{Spec}(R[x,y, x_1, \ldots, x_n]/(xy-b))$ is defined by the ideal $(m_s, x,y, x_1, \ldots, x_n)$ and so the completion of the \'{e}tale local ring of $X$ at $p$ is of the form $R = \mathcal{\hat{O}}_{S,s}^{\text{\'{e}t}}[[x_1, \ldots, x_n]][[x,y]] / (xy - b)$. We may assume $E$ is effective by adding to it the pullback of some effective Cartier divisor on $S$. Locally near $p$ we have that $E$ is given by an element $r\in \text{Frac}(R)^*$. By Theorem 4.1 of \cite{holmes} we may write $r = aux^my^n$, where $a\in \mathcal{O}_{S,s}[[x_1, \ldots, x_n]]$, $u\in R^*$, and $m, n\in \mathbb{Z}_{\geq 0}$. 

Suppose without loss of generality that $y$ vanishes on $v_2$ and $x$ vanishes on $v_1$. As $y$ is generically invertible on $v_1$, at the generic point $\eta_1$ of $v_1$ we have that $E$ is defined by $ax^m$, hence by $ab^m$. Hence $\sigma_1^*E = \text{div}(ab^m)$. Similarly $\sigma_2^*E = \text{div}(ab^n)$. Thus $\sigma_1^*E = \sigma_2^*E + \text{div}b^{m-n}$.
\end{proof}

\section{The Picard scheme and alignment}\label{sec:five}
In this section we define the Picard scheme and an object $\text{Pic}^{0}_{X/S}$ that will serve as a generalisation of the Jacobian of a curve in higher dimensions. We shall then show that a generically smooth semistable morphism $X\rightarrow S$ where $S$ is an excellent regular scheme is aligned if and only if the closure of the unit section in $\text{Pic}^{[0]}_{X/S}\supset \text{Pic}^{0}_{X/S}$ is flat over $S$. This will allow us to show that $\text{Pic}^{0}_{X/S}$ permits a N\'{e}ron model if $X/S$ is aligned.

\begin{definition}
Let $S$ be a scheme, and $X$ be an $S$-scheme. The functor $P_{X/S}: \text{(Sch/S)}^0\rightarrow \text{(Sets)}$ is defined as \[P_{X/S}(T) = \text{Pic}(X\times_S T).\]
The \textit{relative Picard functor of $X$ over $S$} is the $fppf$-sheaf associated to $P_{X/S}$. It is denoted by $\text{Pic}_{X/S}$.
\end{definition}

\begin{remark}\label{formofpic}
By Proposition 8.1.4 of \cite{bosch} when $f: X\rightarrow S$ is quasi-compact and quasi-separated, satisfies $f_*(\mathcal{O}_X) = \mathcal{O}_S$ universally, and $f$ admits a section, there exists an exact sequence
\[0 \rightarrow \text{Pic}(T) \rightarrow \text{Pic}(X\times_ST) \rightarrow \text{Pic}_{X/S}(T)\rightarrow 0\] for any flat $S$-scheme $T$. Thus one obtains an easy way to represent elements of $\text{Pic}_{X/S}(T)$. The condition that $f_*(\mathcal{O}_X) = \mathcal{O}_S$ holds universally is satisfied when $f$ is proper and flat and has geometrically connected and reduced fibres by Exercise 9.3.11 of \cite{fantechi}. In particular, it holds in the case of semistable morphisms, and so in particular semistable morphisms are cohomologically flat in codimension $0$.
\end{remark}

Theorem 8.3.1 of \cite{bosch} informs us that if $f: X \rightarrow S$ is a proper, flat, finitely presented morphism of schemes that is cohomologically flat in dimension $0$, then $\text{Pic}_{X/S}$ is representable by a locally separated algebraic space locally of finite presentation over $S$. Denote by $\text{Pic}^0_{X_U/U}$ the fibre-wise connected component of the identity, and $\text{Pic}^{[0]}_{X/S}$ its closure in $\text{Pic}_{X/S}$.

\begin{theorem}\label{mainresult1}
Let $X\rightarrow S$ be a semistable morphism, where $S$ is an excellent regular scheme. Let $U\subset S$ be a dense open subscheme, and assume $X$ is smooth over $U$. Then the following are equivalent:
\begin{enumerate}
\item $X/S$ is aligned.
\item The closure of the unit section in $\text{Pic}^{[0]}_{X/S}$ is \'{e}tale over $S$.
\item The closure of the unit section in $\text{Pic}^{[0]}_{X/S}$ is flat over $S$.
\end{enumerate}
\end{theorem}

First we give a useful lemma, which itself uses the following result of Holmes:

\begin{lemma}[Lemma 5.17 of \cite{holmes}]\label{sectionsofseparated}
Let $S$ be a Noetherian scheme and $U\subset S$ dense open. Let $f: X \rightarrow S$ be a morphism of schemes locally of finite type and which is an isomorphism over $U$, and such that $f^{-1}U$ is schematically dense in $X$. Let $x\in X$ be a point. The following are equivalent:
\begin{enumerate}
\item $f$ is \'{e}tale at $x$;
\item $f$ is flat at $x$;
\item there exists an open neighbourhood $V$ of $f(x)$ in $S$ and a section $\sigma: V \rightarrow X$ through $x$.
\end{enumerate}
\end{lemma}

\begin{lemma}\label{droppingproj}
Let $S=\text{Spec}(R)$ be Noetherian strictly henselian local scheme with closed point $s$, and let $U\subset S$ a dense open subscheme. Let $f:X\rightarrow S$ be a morphism of algebraic spaces that is locally separated and of finite type, and suppose further that $X/S$ is an isomorphism over $U$. Then the following are equivalent:
\begin{enumerate}
\item $f$ is \'{e}tale at points of the fibre $X_s$.
\item $X/S$ is flat at points of the fibre $X_s$.
\item For every point $x$ in the special fibre there exists a section $\sigma:S\rightarrow X$ through $x$.
\end{enumerate}
\end{lemma}
\begin{proof}
It is immediate that $(1)$ implies $(2)$. We first show $(2)$ implies $(3)$. 

Suppose that $X/S$ is flat at points in the special fibre. As the flat locus is open by [\cite{stacks}, Tag 05WU] and contains $X_s$ and $X_U$, we may replace $X$ by this open subspace and assume $X/S$ is flat. 

By Lemma 67.31.3 of [\cite{stacks}, Tag 0D4L], the fibre dimension is lower semicontinuous for flat morphisms of algebraic spaces of finite presentation, and so as $X/S$ is of dimension $0$ over the dense open subscheme $U$ of $S$, the fibre dimension is $0$ everywhere. Fixing an \'{e}tale presentation $Y\rightarrow X$ of $X$ by an $S$-scheme $Y$, we see that $Y/S$ is of relative dimension $0$. By Lemma 28.28.5 of [\cite{stacks}, Tag 0397] we conclude that $Y/S$ is locally quasi-finite, and so in particular $X/S$ is locally quasi-finite. 

We may now apply Lemma (C.2) of \cite{rydh}, which says that a locally separated, locally quasi-finite algebraic space $X$ over a strictly henselian local scheme $S=\text{Spec}(R)$ is such that every geometric point $x$ in the special fibre $X_s$ admits an affine neighbourhood $\text{Spec}(\mathcal{O}_{X,x})$ with $\text{Spec}(\mathcal{O}_{X,x})\rightarrow S$ a finite morphism. In particular, $x$ admits an affine open neighbourhood. Hence we conclude by \hyperref[sectionsofseparated]{Lemma \ref*{sectionsofseparated}} that there exists a section $\sigma: S \rightarrow X$ through $x$, showing that $(2)$ implies $(3)$

Now suppose $X/S$ admits a section through $x\in X_s$. Choose an \'{e}tale presentation $Y\rightarrow X$ of $X$ by a scheme $Y$. Every point in $Y_s$ has finite residue field over $s$, and hence are closed in $Y_s$. By Lemma 28.19.6 of [\cite{stacks}, Tag 01TC] we conclude that $Y/S$ is locally quasi-finite at points of $Y_s$, and hence so too is $X/S$ locally quasi-finite at points of $X_s$. Lemma (C.2) of \cite{rydh} now implies each point of $X_s$ has an affine neighbourhood. Hence we may apply \hyperref[sectionsofseparated]{Lemma \ref*{sectionsofseparated}} to find that $X/S$ is \'{e}taletw at points of $X_s$. 
\end{proof}

\begin{proof}[Proof of Theorem \ref{mainresult1}]
Properties $(i)$ to $(iii)$ are \'{e}tale local on the $S$, so without loss of generality we may assume that $S$ is strictly henselian local with closed point $s$. 

Suppose first that $X/S$ is aligned. We shall show $(2)$. Let $\text{clo}(e)$ denote the closure of the unit section in $\text{Pic}^{[0]}_{X/S}$, and $p\in \text{clo}(e)$ a point. To show $\text{clo}(e)$ is \'{e}tale over $S$ at $p$ it suffices by \hyperref[droppingproj]{Lemma \ref*{droppingproj}} construct a section from $S$ through $p$ in $\text{clo}(e)$. By \hyperref[traitsexist]{Lemma \ref*{traitsexist}} there exists a non-degenerate trait $\phi: T \rightarrow \text{clo}(e)$ through $p$, and by \hyperref[formofpic]{Remark \ref*{formofpic}} this corresponds to a line bundle $\mathcal{L}$ on $X_T/T$ that is trivial over $\phi_S^*(U)$, where $\phi_S: T \rightarrow S$ is the composition of $\phi$ with the structure map to $S$ and $X_T/T$ is a semistable morphism obtained by pull-back. By choosing a rational section of $\mathcal{L}$ that is trivial over $\phi^*_S(U)$, we obtain a vertical Cartier divisor $D$ on $X_T$ whose restriction to the generic fibre is $0$. 

As $D$ is Cartier we obtain a $T$-Cartier vertex labelling on $\Gamma_s$, and we may assume that it takes the value $0$ at some vertex $v$. By \hyperref[forwardimplication]{Lemma \ref*{forwardimplication}} we obtain a Cartier divisor $\overline{D}$ on $X/S$ that pulls back to $D$ over $T$. This corresponds to a section $\sigma: S \rightarrow \text{clo}(e)$ through $p$, whence we conclude $(2)$ holds.

Clearly $(2)\Rightarrow (3)$, so it remains to show $(3)\Rightarrow (1)$. Equivalently, we'll show that if $X/S$ is not aligned then the closure of the unit section in $\text{Pic}^{[0]}_{X/S}$ is not flat over $S$. Hence we assume that $X/S$ is not aligned at $s$. By \hyperref[reverseimplication]{Lemma \ref*{reverseimplication}} there exists a non-degenerate trait $f:T \rightarrow S$ through $s$ along with a Cartier divisor $D$, trivial on the generic fibre of $X_T$, such that there does not exist a Cartier divisor $\overline{D}$ on $X/S$, trivial over $U$, that pulls back to $D$. By the definition of $\text{Pic}^{[0]}_{X/S}$ we have that $D$ gives a morphism $\phi: T \rightarrow \text{Pic}^{[0]}_{X/S}$, and in fact the image of $T$ lands in $\text{clo}(e)$. To see this last point we note that the generic point of $T$ lands in the image of the unit section and hence $T$ lands in $\text{clo}(e)$ as $\text{clo}(e)$ is closed in $\text{Pic}^{[0]}_{X/S}$.

If $\text{clo}(e)$ is flat over $S$ then by \hyperref[droppingproj]{Lemma \ref*{droppingproj}} there exists a section $\sigma: S \rightarrow \text{clo}(e)$ of $\text{clo}(e)\rightarrow S$ through the point $\phi(t)$, where $t$ is the closed point of $T$. This morphism corresponds to a Cartier divisor $\overline{D}$ on $X/S$, trivial over $U$. As $T$ is reduced we conclude that $\phi:T \rightarrow \text{clo}(e)$ factors through $f:T\rightarrow S$, and hence that the pullback of $\overline{D}$ to $T$ is linearly equivalent to $D$, contradicting how $D$ was chosen.
\end{proof}

Under additional hypotheses on $X/S$, \hyperref[mainresult1]{Theorem \ref*{mainresult1}} allows us to show that alignment of $X/S$ is equivalent to the existence of a N\'{e}ron model for $\text{Pic}^0_{X_U/U}$. When $f: X \rightarrow S$ is a projective semistable morphism, smooth over a dense open subscheme $U\subset S$, then Theorem 8.2.2 of \cite{bosch} informs us that $\text{Pic}_{X/S}$ exists as an $S$-scheme locally of finite presentation over $S$.

\begin{lemma}\label{romagnylemma}
Let $G$ be a group scheme over a base scheme $S$ such that $G\rightarrow S$ is locally of finite presentation and $G\rightarrow S$ is smooth along the unit section. Let $H$ denote the smooth locus of $G\rightarrow S$. Then $H$ is an open and closed subgroup scheme of $G$. 
\end{lemma}
\begin{proof}
Clearly $H$ is open in $G$. We first show it is a subgroup scheme of $G$. Note that the fibres of $G/S$ are group schemes admitting smooth open sets, and as such are smooth everywhere. As $G/S$ is of finite presentation, it follows that the smooth locus of $G/S$ is exactly the flat locus. By the \textit{crit\`{e}re de platitude par fibres} the image of $H\times H\rightarrow G$ in $G$ is flat and hence the image is contained in $H$. As $H$ is closed under inversion it follows that $H$ is a subgroup scheme. 

It remains to show $H$ is closed in $G$. To accomplish this we will show that the \'{e}tale locus of $G/H$ over $S$ is closed in $G/H$, where the quotient $G/H$ exists as an algebraic space over $S$ by Tag 071R of \cite{stacks}. 

Let $(G/H)_{fppf}$ denote the quotient $fppf$-sheaf on $S$ of $H$ acting on $G$. The action of $H$ on $G$ is free, and transitive on the fibres of $G$ over $G/H$, whence we see that $G$ is an $H$-torsor over $(G/H)$ in the category of $fppf$-sheaves on $S$.

Choose an $fppf$-cover $T\rightarrow G/H$ of $G/H$ such that there exists a section from $T$ to $G$. One obtains the following product diagram:

\begin{equation}
\begin{tikzcd}
H\times_s T \arrow[r] \arrow[d]
& G \arrow[d] \\
T \arrow[r]
& (G/H).
\end{tikzcd}
\end{equation}

As smoothness is $fppf$-local on the target, we conclude that $G\rightarrow G/H$ is smooth. From this we also see that sections of $G\rightarrow G/H$ exist \'{e}tale locally on $G/H$, and hence that $G$ is an \'{e}tale $H$-torsor over $G/H$. 

It remains to show that the \'{e}tale locus of $G/H$ over $S$ is closed in $G/H$. First we show $G/H$ is unramified over $S$.

By Theorem 3.2 in Expos\'{e} IVa of \cite{sga3} the quotient $G_s/H_s$ exists as a smooth group scheme over $s$ for any point $s\in S$. Furthermore, if $G_s^0$ denotes the connected component of the identity in $G_s$, the quotient $G_s/G_s^0$ is of dimension $0$ over $s$ by Theorem 5.5.1 of \cite{sga3}, and we conclude by Remark 5.3.2 of \cite{sga3} that $G_s/H_s$ is \'{e}tale over $s$ as $G_s^0\subset H_s$. 

As sheaves on the small \'{e}tale site of $S$ the fibres of the quotient are the quotients of the fibres, and so by Theorem 5.5.1 of Expos\'{e} VIa of \cite{sga3} the fibres of $G/H$ are \'{e}tale schemes over their images in $S$. In particular they are unramified, and so Lemma 28.33.12 of Tag 02G3 of \cite{stacks} informs us that $G/H$ is unramified over $S$.

We conclude by showing the \'{e}tale locus of $G/H$ over $S$ is closed. Choose an \'{e}tale presentation $P\rightarrow G/H$ of $G/H$. $P$ is unramified over $S$, and as being a closed immersion is \'{e}tale local on the base by Tag 02L6 of \cite{stacks} it suffices to show that the \'{e}tale locus of $P\rightarrow S$ is closed.

Let $x\in P$ have image $s\in S$. After taking a sufficiently small \'{e}tale neighbourhood of $s$, Tag 04GL of \cite{stacks} tells us that we can find a Zariski open neighbourhood $U$ of $x$ such that $U\rightarrow S$ is a closed immersion. Thus $U\rightarrow S$ is either \'{e}tale or the \'{e}tale locus is empty. Either way, we find that the \'{e}tale locus is closed, and hence that its pullback $H$ in $G$ is closed.
\end{proof}

\begin{corollary}\label{sufficientsmoothness}
Let $S$ be an excellent regular locally Noetherian scheme, $U\subset S$ a dense open subscheme, and $f: X \rightarrow S$ a projective semistable morphism smooth over $U$. Let $e$ denote the unit section of $\text{Pic}_{X/S}$. If $\text{Pic}_{X/S}$ is smooth over $S$ along $e$, then $\text{Pic}^0_{X/S}$ is a smooth open subgroup scheme of $\text{Pic}_{X/S}$, and the closure $\text{Pic}^{[0]}_{X/S}$ of $\text{Pic}^{0}_{X_U/U}$ in $\text{Pic}_{X/S}$ is smooth over $S$.
\end{corollary}
\begin{proof}
By the assumption that $\text{Pic}_{X/S}$ is smooth along $e$, 15.6.5 of \cite{egaiv_iii} implies that $\text{Pic}^0_{X/S}$ is an open subscheme of $\text{Pic}_{X/S}$. By \hyperref[romagnylemma]{Lemma \ref*{romagnylemma}} the smooth locus of $\text{Pic}_{X/S}$ is an open and closed subgroup scheme. In particular, as $\text{Pic}^0_{X/S}$ is the fibrewise connected component of $e$ it is contained in the smooth locus, and hence is smooth over $S$.

The restriction of $\text{Pic}^0_{X/S}$ over $U$ is an open subscheme of the smooth locus, and so its closure in $\text{Pic}_{X/S}$ is a union of connected components of the smooth locus. In particular $\text{Pic}^{[0]}_{X/S}$ is also smooth over $S$.
\end{proof}

\begin{theorem}[c.f. Theorem 6.2 of \cite{holmes}]
\label{mainresult2}
Let $S$ be an excellent regular locally Noetherian scheme, $U\subset S$ a dense open subscheme, and $f: X \rightarrow S$ a projective semistable morphism smooth over $U$. Let $e$ denote the unit section of $\text{Pic}^{0}_{X_U/U}$ and $\text{clo}(e)$ for the closure of $e$ in $\text{Pic}^{[0]}_{X/S}$. Assume that $\text{Pic}_{X/S}$ is smooth over $S$ along the unit section.
\begin{enumerate}
\item If $X$ is regular and $\text{clo}(e)$ is \'{e}tale over $S$, then a N\'{e}ron model for $\text{Pic}^0_{X_U/U}$ exists.
\item If a N\'{e}ron model for $\text{Pic}^0_{X_U/U}$ exists, then $\text{clo}(e)$ is \'{e}tale over $S$.
\end{enumerate}
\end{theorem}

\begin{lemma}\label{isomoveru}
With the assumptions as in \hyperref[mainresult2]{Theorem \ref*{mainresult2}}, one has $\text{Pic}^{[0]}_{X/S}\times_S U \cong \text{Pic}^0_{X_U/U}$.
\end{lemma}
\begin{proof}
We first show that $\text{Pic}^{[0]}_{X/S}\times_S U$ is quasi-compact over $U$. As quasi-compactness can be checked locally, we may assume for now that $U$ is connected. By Theorem 8.2.5 of \cite{bosch}, we can cover $U$ by open sets $V_i$ such that $\text{Pic}_{X_{V_i}/V_i}$ is a disjoint union of quasi-projective schemes over $V_i$. In particular, $\text{Pic}^{[0]}_{X/S}\times_S V_i$ is the closure of a connected open subscheme in $\text{Pic}_{X_{V_i}/V_i}$, and so is quasi-projective over $V_i$, hence quasi-compact over $V_i$. 

Thus $\text{Pic}^{[0]}_{X/S}\times_S U$ is quasi-compact over $U$. By Theorem 8.4.3 of \cite{bosch}, we conclude that $\text{Pic}^{[0]}_{X/S}\times_S U$ is proper over $U$, and so by Lemma 15.5.4 of \cite{egaiv_iii} its fibres are connected over $U$. Hence we find that the fibres of $\text{Pic}^{[0]}_{X/S}\times_S U$ over $U$ are equal to those of $\text{Pic}^0_{X_U/U}$, and the conclusion follows by the reducedness of both schemes. 
\end{proof}

The proof of \hyperref[mainresult2]{Theorem \ref*{mainresult2}} uses the following result from \cite{holmes}.

\begin{lemma}[Lemma 6.1 of \cite{holmes}]
\label{smoothdescent}
Let $S$ be a scheme, $U \subset S$ a dense open subscheme with $U\rightarrow S$ quasi-compact, and $f: S' \rightarrow S$ a smooth, surjective morphism. Let $A/U$ be an abelian scheme, and suppose $f^*A$ has a N\'{e}ron model $N'$ over $S'$. Then $A$ has a N\'{e}ron model $N$ over $S$, and $f^*N = N'$.
\end{lemma}

\begin{proof}[Proof of Theorem \ref{mainresult2}]
As $\text{clo}(e)$ is flat over $S$, the quotient $N$ of $\text{Pic}^{[0]}_{X/S}$ by $\text{clo}(e)$ exists as a separated group algebraic space, with separatedness following from the unit section of $N$ being a closed immersion. As $N$ permits an $fppf$-cover by $\text{Pic}^{[0]}_{X/S}$, which itself is smooth over $S$ by \hyperref[sufficientsmoothness]{Corollary \ref*{sufficientsmoothness}}, Corollary 6.5.2 of \cite{egaiv_ii} implies that $N$ is smooth. Using the separatedness of $N$ and the reduced-to-separated theorem (10.2.2 of \cite{vakil}), the uniqueness part of the N\'{e}ron mapping property is automatic, so we must show existence.

Let $T\rightarrow S$ be a smooth morphism of algebraic spaces and $T_U\rightarrow N_U$ be an $S$-morphism. By choosing a presentation of $T$ we can find a surjective \'{e}tale morphism $T'\rightarrow T$ where $T'$ is a scheme; by base change we have a morphism $T'_U\rightarrow \text{Pic}^{[0]}_{X/S}\times_S U$.

Because $X/S$ admits a section and is cohomologically flat in dimension $0$, by Proposition 8.1.4 of \cite{bosch} there exists a line bundle $\mathcal{F}$ on $T'_U\times_U X_U$ such that $T'_U\rightarrow \text{Pic}^{[0]}_{X/S}\times_S U$ is given by $\mathcal{F}$. Let $D = \sum_i n_i p_i$ be a Cartier divisor on $T'_U\times_U X_U$ with $\mathcal{O}(D) \cong \mathcal{F}$, where the $p_i$ are prime Weil divisors. Define $\overline{D} =\sum_i n_i \overline{p}_i$, where $\overline{p_i}$ denotes the scheme-theoretic closure of $p_i$ in $T'\times_S X$. We claim $T'\times_S X$ is regular, and hence $\overline{D}$ is a Cartier divisor. This follows as $T'\rightarrow S$ is smooth and $X$ is regular; by [\cite{stacks}, Tag 036D] regularity is local on the base in the smooth topology. Hence $\overline{D}$ is a Cartier divisor. This defines a morphism $T'\rightarrow \text{Pic}^{[0]}_{X/S}\rightarrow N$, whose restriction to $U$ coincides with $T'_U\rightarrow N_U$. This descends to a morphism $T\rightarrow N$ by \'{e}tale descent, necessarily unique by the separatedness of $N$, concluding the proof of $(1)$.

We will now prove $(2)$. Let $N$ be the N\'{e}ron model of $\text{Pic}^0_{X_U/U}$. Theorem 8.2.2 of \cite{bosch} implies that $\text{Pic}_{X/S}$ is a scheme. Note that $\text{Pic}^{[0]}_{X/S}\times_S U = \text{Pic}^0_{X_U/U}$ by \hyperref[isomoveru]{Lemma \ref*{isomoveru}}, and so by the N\'{e}ron mapping property we obtain a morphism $\phi: \text{Pic}^{[0]}_{X/S} \rightarrow N$. If $K$ denotes the kernel of $\phi$, we would like to show that $K$ is flat over $S$ and equal to $\text{clo}(e)$. We shall do this by showing $\phi$ is flat, for if $\phi$ were flat then so too would $K$ be flat over $S$ by property of base change. Furthermore the map $K \rightarrow \text{Pic}^{[0]}_{X/S}$ would be a closed immersion as $N$ is separated over $S$, whence its unit section is a closed immersion. Then, as the preimage of $U$ in $K$ is schematically dense by Theorem 11.10.5 of \cite{egaiv_iii} and $K_U=\text{clo}(e)_U$, we could conclude $K=\text{clo}(e)$. 

To show $\phi$ is flat, it suffices by \hyperref[droppingproj]{Lemma \ref*{droppingproj}} to show that the fibre $\phi_s$ of $\phi$ over $s$ is flat for all $s\in S$ by [\cite{stacks}, Tag 05X0]. In fact, we will show that the restriction $\phi^0_s: \text{Pic}^0_{X_s/s}\rightarrow N^0_s$ of $\phi$ to the connected components of $e$ is flat, as one can then cover $\text{Pic}^{[0]}_{X_s/s}$ by translates of $\text{Pic}^0_{X_s/s}$, the former being a scheme as it is a group algebraic space over a field (see \cite{artin}). The restriction $\phi^0$ of $\phi$ to the fibrewise-connected components of the identity is an isomorphism over $U$. Note that $\text{Pic}^0_{X/U}$ is an open subscheme of $\text{Pic}_{X/S}$ by 8.4 of \cite{bosch}, hence it is flat and quasi-separated over $S$. Hence \hyperref[groupisseparated]{Lemma \ref*{groupisseparated}} implies $\text{Pic}^0_{X/S}$ is separated over $S$, and so we may apply Proposition 3.1(e) of \cite{SGA7} to find that $\phi^0$ is an open immersion, and hence flat.

\end{proof}

\begin{lemma}\label{groupisseparated}
Let $G$ be a group scheme over a base $S$ with connected fibres, and suppose $G/S$ is  quasi-separated and flat over $S$. Then $G/S$ is separated.
\end{lemma}
\begin{proof}
By the quasi-separatedness of $G/S$ we may apply the valuative criterion of separatedness. So suppose we have a commutative diagram

\[ \begin{tikzcd}
K \arrow{r} \arrow[swap]{d} & G \arrow{d} \\%
R \arrow{r}& S
\end{tikzcd}
\]
where $R$ is the spectrum of a valuation ring and $K$ is the spectrum of its fraction field. Suppose furthermore that we have two morphisms $f_i : R\rightarrow G$ for $i=1,2$ making the diagram commute. To show that $f_1=f_2$ it suffices to show that the two morphisms $\text{id}\times f_i : R\rightarrow G_R = G\times_S R$ induced by property of fibre product agree, and so we reduce to the case $S=R$ is the spectrum of a valuation ring and we have two sections $f_i: S \rightarrow G$ that agree on the generic point. Let $s$ be the closed point of $S$.

Let $U\subset G$ be an open subscheme containing a point in the special fibre, and pick $U$ so that it is affine, hence separated over $S$. Then $U$ is dense in all the fibres, as connected group schemes over fields are irreducible by Lemma 38.7.10 of [\cite{stacks}, Tag 047J].

Let $\tau_{f_i(S)}: G\rightarrow G$ denote the translation by $f_i(S)$ morphism, and define $V_i = \tau_{f_i(S)}^{-1}U$. Let $V\subset V_1\cap V_2$ be an affine open set that meets the special fibre. Then $V$ is separated, flat, and surjective over $S$. Let $\pi: V\rightarrow S$ denote the induced morphism to $S$. Observe that it is $fpqc$ as it is quasi-compact (being affine), as well as faithfully flat. Consider the fibre-product diagram

\[ \begin{tikzcd}
V\times_S G \arrow[r, "p_G"] \arrow[d]
& G \arrow[d] \\
V \arrow[r, "\pi"] \arrow[u, bend left, "f_{i,V}"]
& S \arrow[u, bend right, "f_i" right]
\end{tikzcd}
\]

It suffices to show that the sections $f_{i,V}$ agree, as $\pi$ is an $fpqc$-cover of $S$, hence an epimorphism in the category of schemes. If $\Delta: V \rightarrow V\times_S V$ is the diagonal morphism, then we have $\Delta + f_{i,V}(V): V \rightarrow V\times_S G$ is an element of $V\times_S U(G)$. By base change and the assumption that $U$ is separated over $S$, $V\times_S U$ is separated over $V$. By Lemma 11.10.5 of \cite{egaiv_iii} $\pi^{-1}K$ is dense in $V$ as $K$ is dense in $S$, and so the two sections $\Delta+ f_{i,V}$ agree on a dense open set $\pi^{-1}K$. Hence the $f_{i,V}$ are equivalent, and so the sections $f_i$ also agree.
\end{proof}


\newcommand{\etalchar}[1]{$^{#1}$}

\end{document}